\documentclass[draft]{amsart}
\usepackage{amsmath}
\usepackage{amssymb}

\theoremstyle{proclaim}
\newtheorem{theorem}{Theorem}[section]
\newtheorem{lemma}[theorem]{Lemma}
\newtheorem{corollary}[theorem]{Corollary}
\newtheorem{proposition}[theorem]{Proposition}

\theoremstyle{statement}
\newtheorem{remark}[theorem]{Remark}
\newtheorem{definition}[theorem]{Definition}
\newtheorem{example}[theorem]{Example}

\theoremstyle{fancyproclaim}

\numberwithin{equation}{section}

\begin{document}

\title[Hypercontractivity on free quantum groups]{Hypercontractivity of heat semigroups on free quantum groups}

\author[U.~Franz, G.~Hong, F.~Lemeux, M.~Ulrich, H.~Zhang]{Uwe Franz, Guixiang Hong, Fran\c{c}ois Lemeux,  Micha\"el Ulrich {\protect \and} Haonan Zhang}
\address{FRANZ, Laboratoire de Math\'ematiques de Besan\c{c}on, Universit\'e de Bourgogne-Franche-Comt\'e, France}
\email{uwe.franz@univ-fcomte.fr}
\address{HONG, School of Mathematics and Statistics, Wuhan University, Wuhan 430072, China}
\email{guixiang.hong@whu.edu.cn}
\address{LEMEUX, Laboratoire de Math\'ematiques de Besan\c{c}on, Universit\'e de Bourgogne-Franche-Comt\'e, France}
\address{ULRICH, Laboratoire de Math\'ematiques de Besan\c{c}on, Universit\'e de Bourgogne-Franche-Comt\'e, France}
\email{michael.ulrich@univ-fcomte.fr}
\address{ZHANG, Laboratoire de Math\'ematiques de Besan\c{c}on, Universit\'e de Bourgogne-Franche-Comt\'e, France}
\email{haonan.zhang@edu.univ-fcomte.fr}

\begin{abstract} 
In this paper we study two semigroups of completely positive unital self-adjoint maps on the von Neumann algebras of the free orthogonal quantum group $O_N^+$ and the free permutation quantum group $S_N^+$. We show that these semigroups satisfy ultracontractivity and hypercontractivity estimates. We also give results regarding spectral gap and logarithmic Sobolev inequalities.
\end{abstract}

\maketitle

\subsection*{Subject classification}
46L50, 47A30, 47D03

\subsection*{Keywords}
Free quantum group, heat semigroup, hypercontractivity, logarithmic Sobolev inequality

\section*{INTRODUCTION}

Since the 70s, when the word "hypercontractivity" was coined (see \cite{SHK}), it has yielded a fruitful area of Mathematics. Stronger than the classical notion of contractivity, it has been shown that hypercontractivity is strongly linked to a class of inequalities called logarithmic Sobolev inequalities, which in turn have many applications such as in statistical mechanics (see for instance \cite{HS} for the investigation of the Ising model based on log-Sobolev inequalities). With the rise of noncommutative mathematics, hypercontractivity has also been studied in the context of noncommutative $L^p$ spaces, for instance in \cite{OlkZegar}.

The hypercontractivity for semigroups on some cocommutative compact quantum groups  such as von Neumann algebras of discrete groups, e.g.  free products of $\mathbb{Z}_2$, etc., has been recently studied by Junge et al., see \cite{JPPP} and the references therein.

The goal of this paper is to investigate hypercontractivity for semigroups on the free orthogonal quantum group and the free permutation quantum group. Different definitions for a Brownian motion (and hence for a heat semigroup) could be considered on these quantum groups; we will be interested in the $ad$-invariant generating functionals in order to select semigroups that could pretend to the role of heat semigroups.

This paper is only a short introduction to this topic and it is the authors' hope that much more work will be done in this direction. 

\section{COMPACT QUANTUM GROUPS AND HEAT SEMIGROUPS}

\subsection{Compact quantum group: definition}
Compact quantum groups are a generalization of compact groups in the context of noncommutative mathematics. They are defined in the following way:
\begin{definition}
A compact quantum group is a pair $\mathbb{G}=(A,\Delta)$ such that $A$ is a unital $C^*$-algebra and $\Delta:\mathbb{G} \rightarrow \mathbb{G}\otimes\mathbb{G}$ is a comultiplication, i.e.\ it is a unital $*$-algebra homomorphism and it verifies:
$$(\Delta\otimes id)\circ\Delta=(id\otimes\Delta)\circ\Delta$$
and, moreover, the quantum cancellation properties are verified, i.e.\
$$\overline{Lin}[(1\otimes\mathbb{G})\Delta(\mathbb{G})]=\overline{Lin}[(\mathbb{G}\otimes 1)\Delta(\mathbb{G})]=\mathbb{G}\otimes\mathbb{G}$$
where $\overline{Lin}$ is the norm-closure of the linear span.

The $C^*$-algebra $A$ is also noted $C(\mathbb{G})$.
\end{definition}
It is indeed a generalization, because for any compact group $G$, $(C(G),\Delta_G)$ with the comultiplication arising from the group multiplication:
$$\Delta_G:\begin{array}{cc}&C(G)\rightarrow C(G\times G)\simeq C(G)\otimes C(G)\\&f\mapsto ((x,y)\mapsto f(x.y))\end{array}$$
is a compact quantum group. The relevant examples for this article were defined by Wang, see \cite{VDaWa, wang1995, Wan98}:

\begin{example}[Free Orthogonal Quantum Group, see \cite{wang1995}]
Let $N\geq2$ and $C_u(O_N^+)$ be the universal unital $C^*$-algebra generated by the $N^2$ self-adjoint elements $u_{ij}, 1\leq i,j\leq N$ verifying the relations:
\begin{eqnarray*}
  \sum_k u_{ki}u_{kj}&=&\delta_{ij}=\sum_k u_{ik}u_{jk}
\end{eqnarray*}
We define a comultiplication $\Delta$ by setting $\Delta(v_{ij})=\sum_k v_{ik}\otimes v_{kj}$. Then $(C(O_N^+),\Delta)$ is a compact quantum group called the Free Orthogonal Quantum Group. If we impose in addition commutativity, we recover the classical orthogonal group.
\end{example}

\begin{example}[Free Permutation Quantum Group, see \cite{Wan98}]
Let $N\geq2$ and $C(S_N^+)$ be the universal unital $C^*$-algebra generated by $N^2$ elements $u_{ij}, 1\leq i,j\leq N$ such that for all $1\leq i,j\leq N$:
\begin{eqnarray*}
  &&u_{ij}^2=u_{ij}=u_{ij}^*\\
  &&\sum_k u_{ik}=1=\sum_k u_{kj}
\end{eqnarray*}
We define a comultiplication $\Delta$ by setting $\Delta(u_{ij})=\sum_k u_{ik}\otimes u_{kj}$. Then, $(C(S_N^+),\Delta)$ is a compact quantum group called the Free Permutation Quantum Group. If we impose in addition commutativity, we find the classical permutation group.
\end{example}
For $\mathbb{G}=O_N^+,S_N^+$, we denote by ${\rm Pol}(\mathbb{G})$ the $*$-algebra generated by the generators $u_{ij},1\leq i,j\leq N$ and contained in $C(\mathbb{G})$. It has a bialgebra structure by setting:
$$\epsilon(u_{ij})=\delta_{ij}$$ 
It is called the algebra of polynomials of $\mathbb{G}$.

Moreover, every compact quantum group is endowed with a Haar state, i.e.\ a normalized positive functional $h:C(\mathbb{G})\rightarrow \mathbb{C}$ such that $(h\otimes id)\Delta(a)=h(a)1=(id\otimes h)\Delta(a)$ for each $a\in\mathbb{G}$.

The Haar state allows us to define the reduced $C^*$-algebra of a compact quantum group. If $\mathbb{G}$ is a compact quantum group, then we have the GNS representation of its Haar state $h$, i.e.\ a $*$-homomorphism $\pi:{\rm Pol}(\mathbb{G})\rightarrow B(H)$ with $H$ a Hilbert space and $\Omega\in H$ a unit vector, such that $h(x)=<\Omega,\pi(x)\Omega>$ for all $x\in{\rm Pol}(\mathbb{G})$. The reduced $C^*$-algebra $C_r(\mathbb{G})$ is the norm completion of $\pi({\rm Pol}(\mathbb{G}))$ in $B(H)$. In this article, we will always consider the reduced $C^*$-algebra rather than the universal one. The reason for this is that the Haar state is faithful on the reduced $C^*$-algebra. The faithfulness of $h$ is important to define the $L^p$ spaces, which is done as follows. The space $L^\infty(\mathbb{G})=C_r(\mathbb{G})''$ is the von Neumann algebra generated by $C_r(\mathbb{G})$. We define $L^p(\mathbb{G})$ for $1\leq p<\infty$ as the completion of $L^\infty(\mathbb{G})$ for the norm $\|x\|_p=[h((x^*x)^{p/2})]^{1/p}$. We recall here that the Haar state is a trace (i.e.\ $h(ab)=h(ba)$) whenever the compact quantum group is of Kac type, which is the case for the quantum groups $O_N^+$ and $S_N^+$ treated in this paper. See \cite{px03} and the references therein for non-tracial $L^p$-spaces.

Let us now say a few words about corepresentations, for more details and notations we refer to \cite{CFK, FKS}. A corepresentations of a compact quantum group $\mathbb{G}$ is a unitary matrix $v\in\mathcal{M}_k\otimes\mathbb{G}$ such that $(id\otimes\Delta)(v)=v_{12}v_{23}$, it is irreducible if the only scalar matrices that commute with $v$ are multiples of the identity matrix. The set of all (equivalence classes of) irreducible corepresentations is denoted $Irr(\mathbb{G})$. In the case of $O_N^+$ and $S_N^+$, the irreducible corepresentations can be indexed by $\mathbb{N}$ and we denote by $(u^{(s)}_{ij})_{1\leq i,j\leq \dim V_s}$ the coefficients of the $s^{\rm th}$ irreducible corepresentation, $V_s$ being their linear span.

\subsection{Markov semigroups}
In order to investigate hypercontractivity of heat semigroups, one must be able to define heat semigroups on the quantum groups at hand. We recall here for clarity's sake a certain number of important results, without proofs. More on this topic might be found in \cite{CFK}.

We can define L\'evy processes on quantum groups (Definition 2.4 in \cite{CFK}). If $(j_t)_{t\geq0}$ is such a process, then we can associate to it a Markov semigroup $T_t$ by putting $T_t=(id\otimes \phi_t)\circ\Delta$ where $\phi_t=\Phi\circ j_t$ is the marginal distribution of $j_t$. The L\'evy process $(j_t)_t$ is also associated to a generator $L=\left.\frac{d\phi_t}{dt}\right|_{t=0}$ (actually, there is a one-to one correspondence between generators and L\'evy processes, called the Schoenberg correspondence).

It is important to mention the domain of the Markov semigroup. The operator $T_t$ can either be seen as $T_t:C_u(\mathbb{G})\rightarrow C_u(\mathbb{G})$ or as $T_t:C_r(\mathbb{G})\rightarrow C_r(\mathbb{G})$. We will in the sequel take the second definition, due to our use of the reduced $C^*$-algebra. The semigroup is associated to a Markovian generator $T_L:{\rm Pol}(\mathbb{G})\rightarrow{\rm Pol}(\mathbb{G})$ which is defined by $T_L=(id\otimes L)\circ\Delta=\left.\frac{dT_t}{dt}\right|_{t=0}$.

The two semigroups treated in this paper are KMS-symmetric (even GNS-symmetric, which means that $T_L$ and $T_t$ are self-adjoint on $L^2(\mathbb{G},h)$), therefore they extend to $\sigma$-weakly continuous semigroups on the von Neumann algebra $L^\infty(\mathbb{G}) = C_r(\mathbb{G})''$, see, e.g., \cite[Theorem 2.39]{Ci08}.

Now, in the classical case, a heat semigroup is the Markov semigroup associated to a Brownian motion, which is a particular kind of L\'evy process. So if we had a definition of such a Brownian motion on $O_N^+$ or $S_N^+$, we could define a heat semigroup and this semigroup should be naturally privileged in our study. Unfortunately, to define such an object is not an easy matter. In the classical case, Brownian motions are defined on Lie groups via the Laplace-Beltrami operator. On quantum groups, we do not have a differential structure which would allow us to define a quantum analogue to the Laplace-Beltrami operator. Alternative approaches must thus be found.

One way to do so is to use the notion of gaussianity first introduced by Sch\"urmann (as is done for instance in \cite[see especially Section $5.3$]{Ul} to exhibit a Brownian motion on the unitary dual group). This approach nevertheless fails for $S_N^+$, as indicated by \cite[Proposition $8.6$]{FKS}, since there are no gaussian generators on $S_N^+$.

As an alternative, we will be interested in the class of $ad$-invariant generating functionals (see Section $6$ of \cite{CFK}), i.e.\ the functionals invariant under the adjoint action. Linear functionals $L:{\rm Pol}(\mathbb{G})\to\mathbb{C}$ are $ad$-invariant iff there exist numbers $(c_s)_s$ such that $L(u_{ij}^{(s)})=c_s\delta_{ij}$ for $s\in Irr(\mathbb{G})$. They are classified for $O_N^+$ in \cite[Section 10]{CFK} and in \cite[Section $10.4$]{FKS} for $S_N^+$. In the classical case of Lie groups, \cite[Propositions $4.4$, $4.5$]{Li04} shows that $ad$-invariant processes (or, equivalently, conjuguate-invariant processes) on compact simple Lie groups have a generator constituted of the Laplace-Beltrami operator plus a part due to the L\'evy measure. It therefore seems reasonable to define a Brownian motion from within the class of $ad$-invariant functionals and this will be the approach which we will use in this paper.

\subsection{Heat semigroup on the Free Orthogonal Quantum Group}
We will need the definition of Chebyshev polynomials of the second kind.
\begin{definition}
The Chebyshev polynomials of the second kind are the polynomials $U_s$ given by the relation
$$U_s(X)=\sum_{p=0}^{\lfloor s/2\rfloor}(-1)^p\binom{s-p}{p}X^{s-2p}$$
They are an orthonormal family for the scalar product defined via the semicircular measure.
\end{definition}

We recall the following proposition, found in \cite[Proposition $10.3$]{CFK}, showing that $ad$-invariant functionals on $O_N^+$ are classified by pairs $(b,\nu)$ where $b$ is a non-negative real number and $\nu$ a finite measure with support on the interval $[-N,N]$.
\begin{proposition}
The $ad$-invariant generating functional on ${\rm Pol}(O_N^+)$ with characteristic pair  $(b,\nu)$ ($b\geq0$ and $\nu$ a finite measure on $[-N,N]$) acts on the coefficients of unitary irreducible representations of $O_N^+$ as:
$$L(u_{ij}^{(s)})=\frac{\delta_{ij}}{U_s(N)}\left(-bU_s^\prime(N)+\int_{-N}^N \frac{U_s(x)-U_s(N)}{N-x}\nu(dx)\right)$$
for $s\in\mathbb{N}$, where $U_s$ denotes the $s^{th}$ Chebyshev polynomial of the second kind.
\end{proposition}
The generator of the Markov semigroup, which is defined by: $T_L=(id\otimes L)\circ\Delta$, acts as:
$$T_L(u_{ij}^{(s)})=\frac{1}{U_{s}(N)}\left(-bU_s^\prime(N)+\int_{-N}^N\frac{U_s(x)-U_s(N)}{N-x}\nu(dx)\right)u_{ij}^{(s)}$$
The Markov semigroup is given by $T_t=\exp(tT_L)$. Here we will be interested in in the case $b=1$ and $\nu=0$. Indeed, our formula is similar to Hunt's formula in the case of L\'evy processes on Lie groups and it seems natural to take $\nu=0$, since it seems to play a role analogous to the L\'evy measure in Hunt's formula.

Let us now investigate further this Markovian semigroup. We have:
$$L(u_{ij}^{(s)})=-\frac{\delta_{ij}}{U_s(N)}U_s^\prime(N)$$
Therefore, the eigenvalues of $T_L$ are given by:
$$\lambda_s=-\frac{U_s^\prime(N)}{U_s(N)}$$
with eigenspace $V_s=\text{span}\{u_{ij}^{(s)},1\leq i,j\}$ and multiplicity $m_s=(\dim u^{(s)})^2=U_s(N)^2$ (see \cite{CFK}, section $10$). Now, since the leading coefficient of $U_s$ is equal to one, we can write these polynomials with the help of their zeros:
$$U_s(x)=(x-x_1)\ldots(x-x_s)$$
and therefore:
$$\frac{U_s^\prime(x)}{U_s(x)}=\sum_{k=1}^s \frac{1}{x-x_k}$$
for $x\in \mathbb{R}\backslash\{x_1,\ldots x_s\}$. 

The following classical lemma about Chebyshev polynomials will be useful to us in this section and also in the next.
\begin{lemma}
The zeros of $U_s$ are comprised between $-2$ and $2$.
\end{lemma}
\begin{proof}
We will use the fact that the Chebyshev polynomials of the second kind constitute an orthonormal family with regard to Wigner's semicircle law $\frac{1}{\pi}\sqrt{4-x^2}$ on $[-2,2]$.
Let $\in\mathbb{N}$. Let us denote by $S=\{y_1,\ldots,y_l\}$ the set of all zeros of $U_s$ in $(-2,2)$ that have an odd multiplicity. We set $Q=\prod_{k=1}^l(X-x_k)$. It is obvious that $Q$ divides $U_s$. Let us now assume that $\deg Q <s=\deg U_s$. Therefore, we have:
$$\int_{-2}^2Q(x)U_s(x)\frac{1}{\pi}\sqrt{4-x^2}dx=0$$
But the very definition of $Q$ means that the zeros of $U_sQ$ that are in $(-2,2)$ have an even multiplicity, i.e.\ $U_sQ$ has a constant sign on this interval. For the integral to be zero, we must have $U_sQ=0$, which is absurd. Therefore we must have $U_s=Q$ and this proves the lemma.
\end{proof}
We thus have the following lemma:
\begin{lemma}
For $N\geq2$,
$$\frac{s}{N}\leq -\lambda_s=\frac{U_s^\prime(N)}{U_s(N)}=\sum_{k=1}^s\frac{1}{N-x_k}\leq\frac{s}{N-2}$$
where, for $N=2$, we take the convention that $1/0=\infty$.
\end{lemma}
\begin{proof}
The upper bound of $-\lambda_s$ is a consequence of the previous lemma. To obtain the lower bound, note that $N-x_k > 0$ for $1\le k \le n$. Since $\sum_{k=1}^{n} x_k = 0$, we have 
$$-\lambda_s =\sum_{k=1}^{s}\frac{1}{N-x_k} \geq\frac{s^2}{sN-\sum_{k=1}^{s}x_k} =\frac{s}{N}.$$
\end{proof}

\subsection{Heat semigroups on the Free Permutation Quantum Group}
We rely on the results of \cite{FKS} for $S_N^+$. We consider semigroups with generating functionals defined by: 
$$L(u_{ij}^{(s)})=-\frac{\delta_{ij}U_{2s}^\prime(\sqrt{N})}{2\sqrt{N}U_{2s}(\sqrt{N})}$$
We follow the same reasoning as before. The eigenvalues are:
$$\lambda_s=-\frac{U_{2s}^\prime(\sqrt{N})}{2\sqrt{N}U_{2s}(\sqrt{N})}$$
with eigenspace $V_s=\{u_{ij}^{(s)}, 1\leq i,j\le \dim V_s\}$ and multiplicity $m_s=U_{2s}(\sqrt{N})^2$. We find the estimate:
\begin{lemma}For $N\geq4$,
$$\frac{s}{N}\leq -\lambda_s=\frac{1}{2\sqrt{N}}\sum_{k=1}^{2s}\frac{1}{\sqrt{N}-x_k}\leq \frac{s}{\sqrt{N}(\sqrt{N}-2)}$$
where, for $N=4$, we take the convention that $1/0=\infty$. 
\end{lemma}

\section{ULTRACONTRACTIVITY AND HYPERCONTRACTIVITY}

When we need to distinguish the semigroups, we will denote by $T_t^O$ (resp. $T_t^S$) the semigroup we introduced on $O_N^+$ (resp. $S_N^+$)

\subsection{Ultracontractivity}
We say that a semigroup $T_t$ is ultracontractive if it is bounded from $L^2$ into $L^\infty$ for all $t>0$. In the sequel, we will denote by $\|.\|_\infty=\|.\|$ the operator norm and by $\|x\|_p^p=h((x^*x)^{p/2})$ the $p$-norm ($h$ being the Haar state). We will prove the following result:
\begin{theorem}\label{thmultra}
Let $T_t$ be a semigroup on a Kac-type compact quantum group, such that the following assumptions hold:
\begin{itemize}
  \item The subspaces $V_s$ spanned by the coefficients of the irreducible corepresentations $u^s$ are eigenspaces for the generator $T_L$ of the Markov semigroup, i.e.\
  $$T_Lx=\lambda_sx$$
  for $x\in V_s$
  \item We have an estimate of the form $\lambda_s\leq -\alpha s$ for some $\alpha>0$.
  \item We have an inequality of the form:
  $$\|x\|_\infty\leq (\beta s+\gamma)\|x\|_2$$
  for $x\in V_s$, with $\beta,\gamma\geq0$ and $\beta,\gamma$ are independent of $s$.
\end{itemize}
Then, $T_t$ is ultracontractive: $\|T_tx\|_\infty\leq \sqrt{f(t)}\|x\|_2$, where:
$$f(t)=\frac{\beta^2e^{-2\alpha t}(1+e^{-2\alpha t})+2\beta\gamma e^{-2\alpha t}(1-e^{-2\alpha t})+\gamma^2(1-e^{-2\alpha t})^2}{(1-e^{-2\alpha t})^3}.$$
\end{theorem}
\begin{proof}
We have for $x=\sum_s x_s$ with $x_s\in V_s$:
\begin{eqnarray*}
  \|T_tx\|_\infty&\leq&\sum_{s\in\mathbb{N}}\|T_tx_s\|_\infty=\sum_s e^{\lambda_s t}\|x_s\|_\infty \leq\sum_s e^{-\alpha st}\|x_s\|_\infty \\
  &\leq& \sum_s e^{-\alpha st}(\beta s+\gamma)\|x_s\|_2\leq\left(\sum_s(\beta s+\gamma)^2e^{-2\alpha st}\right)^{1/2}\left(\sum_s\|x_s\|_2^2\right)^{1/2}\\
  &=&\sqrt{f(t)}\|x\|_2
\end{eqnarray*}
where we used the Cauchy-Bunyakovsky-Schwarz inequality.

The computation of $f(t)=\sum_s(\beta^2s^2+2\beta\gamma s+\gamma^2)e^{-2\alpha st}$ is done via the classical series:
\begin{eqnarray*}
  \sum_{k\in\mathbb{N}}e^{-\lambda k}&=&\frac{1}{1-e^{-\lambda}}\\
  \sum_{k\in\mathbb{N}}ke^{-\lambda k}&=&\frac{e^{-\lambda}}{(1-e^{-\lambda})^2}\\
  \sum_{k\in\mathbb{N}}k^2e^{-\lambda k}&=&\frac{e^{-\lambda}(1+e^{-\lambda})}{(1-e^{-\lambda})^3}
\end{eqnarray*}
\end{proof}
Let us mention the following nice consequence
\begin{corollary}\label{coro1}
We have for any heat semigroup satisfying the assumptions of Theorem \ref{thmultra}:
$$\|T_tx\|_\infty\leq f(t/2)\|x\|_1,$$
with $f$ the same function as in Theorem \ref{thmultra}.
\end{corollary}
Let us remark that, when $t$ goes to zero, $f(t)$ is equivalent to $1/t^3$. On $\mathbb{R}^d$, the behavior when $t$ goes to zero of a heat semigroup is in $1/t^{d/2}$, as can be seen e.g. in \cite[Property $R_n$, section II.1]{VSC}, so that we have here a behavior as if we were in "dimension" $6$.
\begin{proof}
We are follow the reasoning of \cite[Corollary $3$]{Bia}.
  
The semigroup is self-adjoint on $L^2(\mathbb{G},h)$, since $L(u_{ij}^{(s)})=L(u_{ji}^{(s)})$ is real. So we can dualize the inequality of Theorem \ref{thmultra} to obtain $\|T_tx\|_2\leq\sqrt{f(t)}\|x\|_1$. We can then combine it to get:
$$\|T_tx\|_\infty\leq\sqrt{f(t/2)}\|T_{t/2}x\|_2\leq f(t/2)\|x\|_1$$
\end{proof}

As a consequence of the Theorem, we deduce that the semigroup we considered on the Free Orthogonal Quantum Group is ultracontractive. Indeed, \cite[Proof of Theorem 2.2]{Bra13b} shows that there exists a constant $D$ (depending on $N$) such that:
\begin{equation}\label{eq-DN}
\|x\|_\infty\leq D(s+1)\|x\|_2
\end{equation}  
when $x\in V_s$. Thus we can apply Theorem \ref{thmultra} with $\alpha=1/N$ and $\beta=\gamma=D$.

In the same way, \cite[Theorem 4.10]{Bra13a} shows that there exists a constant $C$ (depending on $N$) such that on $S_N^+$ and for $x\in V_s$, we have:
\begin{equation}\label{eq-CN}
\|x\|_\infty\leq C(2s+1)\|x\|_2
\end{equation}
This means that we can obtain ultracontractivity for our semigroup on $S_N^+$ by applying Theorem \ref{thmultra} with $\alpha=1/N$, $\beta=2C$ and $\gamma=C$.

\subsection{Special cases $O_2^+$ and $S_4^+$}
We can say more in the case of $O_2^+$. We have $U_s(2)=s+1$ and, differentiating the recurrence relation, we get $U^\prime_s(2)=s(s+1)(s+2)/6$. Therefore we know the exact value of the eigenvalues:
$$\lambda_s=-\frac{s(s+2)}{6}$$
If we take up the computations from Theorem \ref{thmultra}, we get better estimates:
$$\|T_tx\|_\infty\leq \sqrt{D^2\sum_s e^{-\frac{s(s+2)}{3}t}(s+1)^2}\|x\|_2$$
Observe now that:
$$\sum_s e^{-\frac{s(s+2)}{3}t}(s+1)^2\leq \sum_s e^{-\frac{s^2}{3}t}(s+1)^2$$
and, moreover,
$$\sum_{s=0}^\infty e^{-\frac{s^2}{3}t}\leq 1+\sum_{s=1}^\infty se^{-\frac{s^2}{3}t}\leq 1+\sum_{s=1}^\infty s^2e^{-\frac{s^2}{3}t}$$
This yields the inequality:
$$\|T_tx\|_\infty\leq \sqrt{g(t)}\|x\|_2\text{ with }g(t)=4D^2\sum_{s=1}^\infty s^2 e^{-\frac{s^2}{3}t}+D^2$$
The function $s\mapsto s^2e^{-\frac{s^2t}{3}}$ is decreasing on $[\sqrt{\frac{3}{t}},+\infty[$ and increasing on $[0,\sqrt{\frac{3}{t}}]$. Let's set $s_0=\sqrt{\frac{3}{t}}$. For fixed $t$, we have:
\begin{eqnarray*}
  \int_0^{s_0}s^2e^{-\frac{s^2t}{3}}ds&\leq&\sum_{s=1}^{s_0}s^2e^{-\frac{s^2t}{3}}\leq \int_0^{s_0}s^2e^{-\frac{s^2t}{3}}ds+\frac{3}{et}\\
  \int_{s_0}^\infty s^2e^{-\frac{s^2t}{3}}ds&\leq&\sum_{s=s_0}^\infty s^2e^{-\frac{s^2t}{3}}\leq \frac{3}{et}+\int_{s_0}^\infty s^2e^{-\frac{s^2t}{3}}ds
\end{eqnarray*} 
We do the change of variable $u=s\sqrt{t/3}$:
 \begin{eqnarray*}
   \left(\frac{3}{t}\right)^{3/2}\int_0^1 u^2e^{-u^2}du&\leq&\sum_{s=1}^{s_0}s^2e^{-\frac{s^2t}{3}}\leq \left(\frac{3}{t}\right)^{3/2}\int_0^1 u^2e^{-u^2}du+\frac{3}{et}\\
   \left(\frac{3}{t}\right)^{3/2}\int_1^\infty u^2e^{-u^2}du&\leq&\sum_{s=s_0}^\infty s^2e^{-\frac{s^2t}{3}}\leq \frac{3}{et}+\left(\frac{3}{t}\right)^{3/2}\int_1^\infty u^2e^{-u^2}du
 \end{eqnarray*}
 And by combining:
\begin{gather*}
  \left(\frac{3}{t}\right)^{3/2}\int_0^\infty u^2e^{-u^2}du \leq
  \frac{3}{et}+\sum_{s=0}^\infty s^2e^{-\frac{s^2t}{3}}
\leq2\frac{3}{et}+\left(\frac{3}{t}\right)^{3/2}\int_{0}^\infty u^2e^{-u^2}du
\end{gather*}
In other words, when $t$ goes to zero, $g(t)$ behaves like $t^{-3/2}$, and, in the spirit of the remark following Corollary \ref{coro1}, this yields a "dimension" $3$ for the semigroup.

The same reasoning for $S_4^+$ yields the eigenvalues $\lambda_s=-\frac{s(s+1)}{6}$ and the "dimension" of the semigroup on $S_4^+$ is also $3$.

\subsection{Hypercontractivity}

\begin{definition}
We say that a semigroup $T_t$ is hypercontractive if for each $2<p<\infty$, there exists a $\tau_p>0$ such that for all $t\geq \tau_p$ we have:
\begin{equation}\label{hyperineq}
\|T_tx\|_p\leq\|x\|_2
\end{equation}
\end{definition}
Let us remark that if the semigroup $T_t$ is hypercontractive, then the inequality (\ref{hyperineq}) is also true for $1\leq p\leq2$ because for such a $p$ and for any $x\in C(\mathbb{G})$ we have $\|T_t x\|_p \leq \|x\|_p \leq \|x\|_2$. We can also notice that due to duality, we have:
$$\|T_tx\|_2\leq \|x\|_q$$
for $t\geq\tau_p$ and $q$ such that $1/p+1/q=1$. Therefore, for $t$ big enough, $T_t$ is also a contraction from $L^q$ to $L^2$ for any $1<q<2$.

Denote by $D_N$ and $C_N$ the constants from the inequalities \eqref{eq-DN} and \eqref{eq-CN}, respectively. We know from \cite[Proof of Theorem 2.2]{Bra13b} and \cite[Theorem 4.10]{Bra13a} that $D_N\ge 1$, $C_N\ge 1$.

\begin{theorem}\label{thmhyper}
	The semigroup $T_t^O$ (resp. $T_t^S$) we consider on $O_N^+$ (resp. $S_N^+$) is hypercontractive.
\end{theorem}
\begin{proof}
	We use the inequality
	$$\|x\|_p^2\leq \|h(x)1\|_p^2+(p-1)\|x-h(x)1\|_p^2 \qquad x\in L^\infty(\mathbb{G})$$
	for $2<p<\infty$, shown in \cite[Theorem 1]{RX14}. It can indeed be applied in our setting, with $L^\infty(\mathbb{G})=C_r(\mathbb{G})''$ a von Neumann algebra and $h$ a faithful, finite normal trace on it.
	We will write $x=h(x)1+\sum_{s\geq1}x_s$ with $x_s\in V_s$. 
	We have $h(T_t(x))1=T_t(h(x)1)$ because the $V_s$ are eigenspaces for $T_t$. Therefore,
	\begin{eqnarray*}
		\|T_t(x)\|_p^2&\leq&\|T_t(h(x)1)\|_p^2+(p-1)\|T_t(x-h(x)1)\|_p^2\\
		&\leq&|h(x)|^2+(p-1)\left(\sum_{s\geq1}\|T_t(x_s)\|_p\right)^2\\
		&\leq&|h(x)|^2+(p-1)\left(\sum_{s\geq1}e^{\lambda_s t}\|x_s\|_p\right)^2\\
		&\leq&|h(x)|^2+(p-1)\left(\sum_{s\geq1}e^{\lambda_st}\|x_s\|_\infty\right)^2\\
		&\leq&|h(x)|^2+(p-1)\left(\sum_{s\geq1}e^{\lambda_st}(\beta s+\gamma)\|x_s\|_2\right)^2\\
		&\leq&|h(x)|^2+(p-1)\sum_{s\geq1}\left((\beta s+\gamma)e^{\lambda_s t}\right)^2\sum_{s\geq1}\|x_s\|_2^2 \leq \|x\|_2^2
	\end{eqnarray*}
	for $t\geq \tau_p$ with $\tau_p$ such that:
	$$(p-1)\sum_{s\geq1}(\beta s+\gamma)^2e^{2\lambda_s \tau_p}\leq1.$$
\end{proof}

\begin{proposition}
	Hypercontractivity is achieved for $T^O_t$ at least from the time $\tau_p^{(O)}$ given by
	\begin{eqnarray*}
		\tau_p^{(O)}=-\frac{N}{2}\log X,\\
		\text{where }X\text{ is the smallest real positive root of }\frac{X^3-3X^2+4X}{(1-X)^3}&=&\frac{1}{(p-1)D_N^2}.
	\end{eqnarray*}
	Hypercontractivity is achieved for $T^S_t$ at least from the time $\tau_p^{(S)}$ given by
	\begin{eqnarray*}
		\tau_p^{(S)}=-\frac{N}{2}\log Y,\\
		\text{where }Y\text{ is the smallest real positive root of }\frac{Y^3-2Y^2+9Y}{(1-Y)^3}&=&\frac{1}{(p-1)C_N^2}.
	\end{eqnarray*}
\end{proposition}
\begin{proof}
	We use the expression:
	$$(p-1)\sum_{s\geq1}(\beta s+\gamma)^2e^{2\lambda_s \tau_p}=1$$
	drawn from the proof of Theorem \ref{thmhyper}. The precise value of the eigenvalues is too cumbersome to compute, therefore we use a minoration of them:
	\begin{eqnarray*}
		\lambda_s&\leq&-\frac{s}{N}\quad \text{ for } \quad O_N^+,\\
		\lambda_s&\leq&-\frac{s}{N}\quad \text{ for } \quad S_N^+.
	\end{eqnarray*}
	By then setting $X=\exp(-\frac{2\tau_p^{(O)}}{N})$ and $Y=\exp(-\frac{2\tau_p^{(S)}}{N})$ and using the classical series that were already used in the proof of Theorem \ref{thmultra}, we obtain the desired equation for $X$ and $Y$. The fact that the root must be the smallest real one (indeed one can easily check that there exists at least one root between 0 and 1) comes from the fact that we need to take the biggest time $\tau_p$ such that the inequalities
	\begin{eqnarray*}
		\frac{X^3-3X^2+4X}{(1-X)^3}&\leq& \frac{1}{(p-1)D_N^2},\\
		\frac{Y^3-2Y^2+9Y}{(1-Y)^3}&\leq&\frac{1}{(p-1)C_N^2},
	\end{eqnarray*}
	are verified always for $t\geq\tau_p$. But $X$ and $Y$ diminish when the time increases. Therefore we need to choose the smallest positive root.
\end{proof}

For $p\geq 4-\varepsilon_0$ with $\varepsilon_0$ a nonnegative constant, we can obtain a better estimate of $\tau_p^{(O)}$ (resp. $\tau_p^{(N)}$) as the following theorem shows:

\begin{theorem}\label{thmhyper2}
There exists $\varepsilon_0 \ge 0$ such that for any $p\ge 4-\varepsilon_0$,
$$\Vert T_t^O \Vert_{2\to p}\le 1, \quad \text{ for all } \quad t\ge \frac{cN}{2}\log (p-1)+(1-\frac{2}{p})N\log D_N,$$
$$\Vert T_t^S \Vert_{2\to p}\le 1, \quad \text{ for all } \quad t\ge \frac{dN}{2}\log (p-1)+(1-\frac{2}{p})N\log C_N,$$
with $c=\frac{2 \log(\sqrt{3}+1)}{\log 3} \approx 1.8297\ldots$, and $d=\frac{\log(11+\sqrt{105})-\log 2}{\log{3}}\approx 2.15096\ldots$.
\end{theorem}
\begin{proof}
We only prove this theorem for $O_N^+$, the case for $S_N^+$ is similar.
We use again the inequality
$$\|x\|_p^2\leq \|h(x)1\|_p^2+(p-1)\|x-h(x)1\|_p^2 \qquad x\in L^\infty(\mathbb{G})$$
for $2<p<\infty$, shown in \cite[Theorem 1]{RX14}. By the H\"older inequality, for $p \ge 1$, $$\|x_s\|_p\leq \left(D_N \left(s+1\right) \right)^{1-\frac{2}{p}} \|x_s\|_2,x_s\in V_s.$$ Therefore,
\begin{eqnarray*}
  \|T_t(x)\|_p^2&\leq& |h(x)|^2+(p-1)\left(\sum_{s\geq1}e^{\lambda_s t}\|x_s\|_p\right)^2\\
  &\leq&|h(x)|^2+(p-1)\left(\sum_{s\geq1}e^{\lambda_s t}\left(D_N \left(s+1\right) \right)^{1-\frac{2}{p}} \|x_s\|_2\right)^2\\
  &\leq&|h(x)|^2+(p-1)\sum_{s\geq 1} e^{2\lambda_s t} \left(D_N \left(s+1\right) \right)^{2\left(1-\frac{2}{p}\right)} \sum_{s\geq 1}\|x_s\|_2^2.
\end{eqnarray*}
When $t\geq \frac{cN}{2}\log (p-1)+\left(1-\frac{2}{p}\right)N\log D_N$ and $s\geq 1$, we have
\begin{eqnarray*}
  2\lambda_s t&\leq& -cs\log (p-1)-2\left(1-\frac{2}{p}\right)s\log D_N\\
  &\leq& -cs\log (p-1)-2\left(1-\frac{2}{p}\right)\log D_N,
\end{eqnarray*}
and $e^{2\lambda_s t}\leq (p-1)^{-cs}D_N^{-2\left(1-\frac{2}{p}\right)}.$
So it suffices to show that for some $\varepsilon_0 \geq 0$, for any $p\geq 4-\varepsilon_0$,
$$R_p:=\sum_{s\geq 1}\phi_s(p) =\sum_{s\geq 1}(p-1)^{1-cs}(s+1)^{2\left(1-\frac{2}{p}\right)}\leq 1.$$
An easy computation implies that $\phi'_s(p) \leq 0$ if and only if
$$\frac{4(p-1)}{p^2}\leq \frac{cs-1}{\log(s+1)}.$$
Note that $f_1(p)=\frac{4(p-1)}{p^2}$ is decreasing for $p\geq 2$, and $f_2(s)=\frac{cs-1}{\log(s+1)}$ is increasing for $s\geq 1$, thus from $c=\frac{2\log(\sqrt{3}+1)}{\log 3} \approx 1.83 > 1.69 \approx 1+\log 2$ we deduce that
$$f_1(p) \leq f_1(2)=1 < \frac{c-1}{\log 2}=f_2(1)\leq f_2(s), \text{ for all } p \geq 2,s\geq 1.$$
Hence each $\phi_s$ is decreasing for $p\geq 2$, and $R_p$ is also decreasing for $p\geq 2$. Since 
$$R_4 =\sum_{s\geq 1}\frac{s+1}{3^{cs-1}} =\frac{3(2\cdot 3^{c}-1)}{(3^{c}-1)^2}=1,$$
we have $R_p\leq 1$ for all $p\geq 4$. So there exists $\varepsilon_0 \geq 0$ such that $\Vert T_t^O \Vert_{2\to p}\leq 1$.
\end{proof}

\begin{remark}
We can see from \cite[Proof of Theorem 2.2]{Bra13b} that 
$$D_N\leq (1-q^2)^{-1} \prod_{s=1}^{r}\left(1-q^{2s}\right)^{-3},$$
with $r\geq 1$ and $N=q+\frac{1}{q},0<q<1$. Thus we deduce that $\lim\limits_{N\to +\infty}N\log D_N = 0$, which implies that the latter part of $\tau_p^{O}$ in Theorem \ref{thmhyper2}, $(1-\frac{2}{p})N\log D_N$, disappears as $N\to +\infty$. Indeed, it suffices to show that:
$$\lim\limits_{q\to 0} \frac{1}{q}\log\left( (1-q^2)\prod_{s=1}^{r}(1-q^{2s})^3\right)=0.$$
So it is done when we prove for all $s\geq 1$:
$$\lim\limits_{q\to 0} \frac{1}{q}\log(1-q^{2s})=0.$$
This is clear, since $\lim\limits_{q\to 0}\frac{1}{q^{2s}}\log(1-q^{2s})=-1$ for all $s \geq 1$.

We have not been able to proof a similar result for $S_N^+$, since by \cite[Theorem 4.10]{Bra13a}, $C_N\to +\infty$ as $N\to +\infty$.
\end{remark}

\section{FURTHER PROPERTIES OF THE SEMIGROUPS}

We will note ${\rm Pol}(\mathbb{G})_+$ the subset of ${\rm Pol}(\mathbb{G})$ consisting of all such $x$ such that $|x|=x$.
\subsection{Spectral gap}
\begin{definition}
We say that $T_t$ verifies a spectral gap inequality with constant $m>0$ if we have for all $x\in {\rm Pol}(\mathbb{G})_+$:
$$m\|x-h\left(x\right)\|_2^2\leq -h\left(xT_Lx\right)$$
\end{definition}
\begin{proposition}
Our semigroup $T_t^O$ on $O_N^+$ verifies the spectral gap inequality with constant $m=\frac{1}{N}$.
\end{proposition}
\begin{proof}
The eigenvalues of the generator $T_L$ are of the form:
$$\lambda_s =-\frac{U^\prime_s\left(N\right)}{U_s\left(N\right)}=-\sum_{i=1}^s\frac{1}{N-x_i},$$
and we have shown that $\lambda_s \geq \frac{s}{N}$ for $N\geq 2$.

Let us now write $x=\sum_sx_s$. We then get:
$$h\left(xT_Lx\right)=\sum_{s}-\frac{U^\prime_s\left(N\right)}{U_s\left(N\right)}\|x_s\|_2^2$$
Using the fact that the $V_s$ are in orthogonal direct sum, we deduce that: $-h\left(xT_Lx\right)\geq\frac{1}{N}\|x\|_2^2$.

But, we also see that $\|x-h\left(x\right)\|_2\leq\|x\|_2$ and thus we finally get:
$$\|x-h\left(x\right)\|_2^2\leq -Nh\left(xT_Lx\right).$$
\end{proof}
We can prove the following in the same way:
\begin{proposition}
Our semigroup $T_t^S$ on $S_N^+$ verifies the spectral gap inequality with constant $m=\frac{1}{N}$.
\end{proposition}

\subsection{Logarithmic Sobolev Inequalities}

Hypercontractivity is equivalent to Logarithmic Sobolev inequalities, or, shorter, log-Sobolev inequalities, see, e.g., \cite{Gross} or \cite[Theorem 3.8]{OlkZegar}. We derive here a log-Sobolev inequality for the generators of our heat semigroups. There is nothing new in this Section, we include it only for comparison.

\begin{proposition}\label{ineg}
There exists a constant $t_0>0$, such that, if we denote $q\left(t\right)=\frac{4}{2-t/t_0}$, we then have for $0\leq t\leq t_0$:
$$\|T_t^\mathbb{G}:L^2\rightarrow L^{q\left(t\right)}\|\leq 1$$
where $\mathbb{G}=O_N^+$ or $S_N^+$
\end{proposition}
\begin{proof}
We take for $t_0$ the optimal time for hypercontractivity $T_t:L^2\rightarrow L^4$, then we have $q(t_0)=4$, and $T_0= {\rm Id}:L^2\to L^2$ and $T_{t_0}:L^2\to L^4$ are contractions. The Proposition therefore follows by Stein interpolation.
\end{proof}

\begin{theorem}
For $x\in L^\infty(\mathbb{G})_+\cap D(T^{\mathbb{G}}_L)$ and with the same assumptions as in Proposition \ref{ineg}, we have the following inequality:
$$h\left(x^2\log x\right)-\|x\|_2^2\log\|x\|_2\leq -\frac{c}{2}h\left(xT_Lx\right)$$
where $c=t_0/2$.
\end{theorem}
\begin{proof}

We define : $F\left(t\right)=\|x_t\|_{q\left(t\right)}$, where we note $x_t=T_tx$. Because of Proposition \ref{ineg}, we know that $\log F\left(t\right)\leq \log F\left(0\right)$. Hence:
$$\frac{d}{dt}\log F\left(t\right)_{|t=0}\leq0$$
As in \cite[Lemma 3.7]{OlkZegar} this term is given by:
\begin{eqnarray*}
  \frac{d}{dt}\log\|x_t\|_q &=& \frac{d}{dt}\left(\frac{1}{q}\log\|x_t\|_q^q\right)\\
  &=& -\frac{\dot{q}}{q}\log\|x_t\|_q+\frac{1}{q\|x_t\|_q^q}\frac{d}{dt}\|x_t\|_q^q \\
  &=&  -\frac{\dot{q}}{q}\log\|x_t\|_q+\frac{1}{q\|x_t\|_q^q}\left(q h\big(x_t^{q-1} T_L(x_T)\big) + \dot{q} h\big(x_t^q\log(x_t)\big)\right).
\end{eqnarray*}
From this we obtain the desired inequality, because $q\left(0\right)=2$, $\dot{q}\left(0\right)=2/t_0$.
\end{proof}

\section{CONCLUSION}

We have studied in this paper two Markov semigroups, one on $O_N^+$ and the other on $S_N^+$, which could be candidates for a Brownian motion on these quantum groups. We have shown that these semigroups are hypercontractive and satisfy log-Sobolev inequalities.

Several natural questions are: What can be said about other semigroups on $O_N^+$ or $S_N^+$? What are the optimal times for hypercontractivity? What happens on other quantum groups, e.g.\ $SU_q(2)$, which are not Kac-type?

\section*{Acknowledgements} 
UF, MU and HZ are supported by MAEDI/MENESR and DAAD through the PROCOPE programme, and by MAEDI/MENESR and the Polish MNiSW through the POLONIUM programme. GH is supported by MINECO: ICMAT Severo Ochoa project SEV-2011-0087.


\begin{thebibliography}{99}

\bibitem{BCL94}
Keith Ball, Eric~A. Carlen, and Elliott~H. Lieb,
Sharp uniforme convexity and smoothness inequalities for trace norms,
\textit{Invent. Math.}, \textbf{115}(1994), 463--482.

\bibitem{Bia}
Philippe Biane,
Free hypercontractivity,
\textit{Comm. Math. Phys.}, \textbf{184}(1997), 457--474.

\bibitem{BZ00}
Thierry Bodineau and Boguslaw Zegarlinski,
Hypercontractivity via spectral theory,
\textit{Infin. Dimens. Anal. Quantum. Probab. Rel. Top.}, \textbf{3}(2000), 15--31.

\bibitem{Bra13a}
Michael Brannan,
Reduced operator algebras of trace-preserving quantum automorphism groups,
\textit{Doc. math.}, \textbf{18}(2013), 1349--1402.

\bibitem{Bra13b}
Michael Brannan,
{S}trong asymptotic freeness for free orthogonal quantum groups,
\textit{Canadian Math. Bulletin}, \textbf{57}(2014), 708--720.

\bibitem{CS08}
Raffaella Carbone and Emanuela Sasso,
Hypercontractivity for a quantum {O}rnstein-{U}hlenbeck semigroup,
\textit{Probab. Theory Related Fields}, \textbf140(2000), 15--31.

\bibitem{Ci08}
Fabio Cipriani,
\textit{Dirichlet forms on noncommutative spaces},
In: ``Quantum Potential Theory'', Uwe Franz and Michael Sch{\"u}rmann (eds), {\em Lecture Notes in Mathematics} Vol.\ 1954, pp.\ 161--276, Springer, 2008.

\bibitem{CFK}
Fabio Cipriani, Uwe Franz, and Anna Kula,
\textit{Symmetries of {L}{\'e}vy processes on compact quantum groups, their
  {M}arkov semigroups and potential theory},
\textit{J. Funct. Analysis}, \textbf{266}(2014), 2789--2844.

\bibitem{VDaWa}
Alfons~Van Daele and Shuzhou Wang,
Universal quantum groups,
\textit{Int. J. Math.}, \textbf{07}(1996), 255--263.

\bibitem{FKS}
Uwe Franz, Anna Kula, and Adam Skalski,
{L}\'evy {P}rocesses on {Q}antum {P}ermutation {G}roups,
In: ``Noncommutative Analysis, Operator Theory and Applications'', Fabio Cipriani, Fabrizio Colombo, Irene Sabadini (eds.), Birkh\"auser, 2016, see also arxiv:1510.08321.

\bibitem{Gross}
Leonard Gross,
Hypercontractivity and logarithmic Sobolev inequalities for the Clifford Dirichlet form,
\textit{Duke Math. J.}, \textbf{42}(1975), 383--396.

\bibitem{HS}
Richard Holley and Daniel Stroock,
Logarithmic {S}obolev inequalities and stochastic {I}sing models,
\textit{J. Stat. phys.}, \textbf{46}(1978),  1159--1194.

\bibitem{JPPP}
Marius Junge, Carlos Palazuelos, Javier Parcet, and Mathilde Perrin,
Hypercontractivity in group von Neumann algebras,
\textit{preprint arXiv:1304.5789}, 2013.

\bibitem{Li04}
Ming Liao,
\newblock {\em L\'evy processes in {L}ie groups}, volume 162 of {\em Cambridge
  Tracts in Mathematics},
\newblock Cambridge University Press, Cambridge 2004.

\bibitem{OlkZegar}
Robert Olkiewicz and Boguslaw Zegarlinski,
Hypercontractivity in noncommutative ${L}^p$ spaces,
\textit{J. Funct. Analysis}, \textbf{161}(1999), 246--285.

\bibitem{px03}
Gilles Pisier and Quanhua Xu,
\textit{Non-commutative Lp-spaces}, Handbook of the geometry of Banach spaces, Vol. 2, 1459–1517, North-Holland, Amsterdam 2003.  

\bibitem{RX14}
Eric Ricard and Quanhua Xu,
A noncommutative martingale convexity inequality,
\textit{Ann. Prob.} \textbf{44}(2016), 867-882.

\bibitem{SHK}
B.~Simon and R.~Hoegh-Krohn,
Hypercontractive semi-groups and two dimensional self-coupled {B}ose
  fields,
\textit{J. Funct. Analysis}, \textbf{9}(1972), 121--180.

\bibitem{Ul}
Micha\"el Ulrich,
Construction of a free {L}\'evy process as high-dimensional limit of
  a {B}rownian motion on the unitary group,
\textit{Infin. Dim. Anal. Quant. Prob.}, \textbf{18}(2015), 1550018.

\bibitem{VSC}
N.~Th. Varopoulos, L.~Saloff-Coste, and T.~Coulhon,
\textit{Analysis and Geometry on Groups},
Cambridge University Press, Cambridge 1992.

\bibitem{ver077}
Roland Vergnioux,
The property of rapid decay for discrete quantum groups,
\textit{J. Operator Theory}, \textbf{57}(2007), 303--324.

\bibitem{wang1995}
Shuzhou Wang,
Free products of compact quantum groups,
\textit{Comm. Math. Phys.}, \textbf{167}(1995), 671--692.

\bibitem{Wan98}
Shuzhou Wang,
Quantum symmetry groups of finite spaces,
\textit{Comm. Math. Phys.}, \textbf{195}(1998), 195--211.

\bibitem{zhang}
Haonan Zhang, A Noncommutative Martingale Convexity Inequality and Its Application to Hypercontractivity, Master's Thesis, University of Franche-Comt\'e, Besan\c{c}on, 2016.

\end{thebibliography}
\end{document}